\documentclass[11pt]{amsart}
\usepackage[margin=3cm]{geometry}

\newtheorem{Thm}{Theorem}

\newtheorem{lem}{Lemma}
\newtheorem{cor}{Corollary}
\newtheorem{prop}{Proposition}
\newtheorem{Problem}{Problem}

\newcommand{\Ocal}{\mathcal{O}}

\begin{document}
\date{June 24, 2016}

\title{Products of two proportional primes}
\author{Pieter Moree and Sumaia Saad Eddin}

\maketitle
{\def\thefootnote{}
}

\begin{abstract}
\noindent 
In RSA cryptography numbers of the form $pq$, with $p$ and $q$ two distinct proportional primes play an important role. For a fixed real number $r>1$ we formalize this by saying that an integer
$pq$ is an RSA-integer if $p$ and $q$ are primes satisfying $p<q\le rp$.
Recently Dummit, Granville and Kisilevsky  showed that substantially more than a quarter of the odd integers
of the form $pq$ up to $x$, with $p, q$
both prime, satisfy $p\equiv q\equiv 3\pmod{4}$. 
In this paper we investigate this phenomenon for RSA-integers. 
We establish an analogue of a strong form of the prime number theorem with the
logarithmic integral replaced by a variant.
{}From this we derive
an asymptotic formula for the number of RSA-integers $\le x$ which is much more precise
than an earlier one derived by Decker and Moree in 2008.
\end{abstract}

\section{Introduction}
Let $\omega(n)$ and $\Omega(n)$ denote the number of distinct, respectively total number 
of prime factors of $n$. Put
$$\pi(x,k)=\sum_{n\le x\atop \omega(n)=k}1{\rm \quad and\quad}N(x,k)=\sum_{n\le x\atop \Omega(n)=k}1.$$
The following asymptotic formula is due to Landau
\cite[p. 211]{Landau}:
\begin{equation}
\label{gauss}
\pi(x,k)\sim N(x,k)\sim \frac{x}{\log x}\cdot \frac{(\log \log x)^{k-1}}{(k-1)!}.
\end{equation}
For a nice survey on $\pi(x,k)$ and $N(x,k)$ up to 1987 see 
Hildebrand \cite{H}. A recent contribution to the study of $\pi(x,k)$
is the discovery of {\bf bias}. Define
$$r(x):=\#\{pq\le x:p\equiv q\equiv 3({\rm mod~}4)\}/\frac{1}{4}\#\{pq\le x\}.$$
(Here and in the
sequel the notation $p$ and $q$ is exclusively used to indicate
primes.) Numerically it seems that consistently $r(x)>1$.
We have, e.g., $r(10^6)\approx 1.183$ and 
$r(10^7)\approx 1.162$.
Dummit et al.~\cite{D.G.K} showed that
\begin{equation}
\label{dumm}
r(x)=1+\frac{(\beta+o(1))}{\log \log x},
\end{equation}
with $\beta \approx 0.334$. This turns out to be in pretty
good agreement with the observed values. 
Since
$\beta/\log \log x $ tends to zero so slowly, \eqref{dumm} 
shows that
substantially more than a quarter of the integers $pq\le x$ satisfy 
$p\equiv 3({\rm  mod~}4)$ and $q\equiv 3({\rm  mod~}4)$.
\subsection{RSA-integers}
In the RSA cryptosystem, see \cite[Chapter 3]{Gathen}, integers 
of the form $n=p\cdot q$ are the main actors. The security
of this system is based on  the current difficulty of factoring 
such integers (sometimes called quasiprimes) in a reasonable time.
As soon as a working quantum computer is developed, it will be
the end of the RSA cryptosystem \cite{Shor}.
The RSA cryptosystem is known to be more easily breakable under
certain special restrictions on $p$ and $q$. E.g., if 
$|p-q|$ is small or if one of $p$ and $q$ is much smaller than 
the other (``unbalanced RSA''). In RSA practice $p$ and $q$
are taken to be proportional, i.e. $p<q<rp$ for some $r>1$.
This does not exclude $q-p$ from being small, but for our
counting purposes this suffices.
\subsection{Bias of RSA-integers}
We study two problems in this paper. One is to 
determine to what extent RSA-integers 
are biased.
If they are, we would glean a very small amount of information about 
their prime factorisation (provided they are generated randomly), and so
the question is somewhat relevant. The other problem is 
to find a precise asymptotic for the
counting function of RSA-integers
$$C_r(x):=\# \left\{pq \leq x: \ p< q \leq rp\right\},$$
where $r>1$ is an arbitrary real fixed number.\\
\indent Note that $C_r(x)$ is
the RSA-analogue of $\pi(x,2)$. Theorem \ref{DM} gives the asymptotic 
behaviour of $C_r(x)$. Comparison with
\eqref{gauss} shows that there are much fewer RSA-integers 
than integers having
two (distinct) prime factors.
\begin{Thm}[Decker and Moree \cite{Decker-Moree}]
\label{DM}
Let $r>1$ be a real number.
As $x$ tends to infinity we have
\begin{equation*}
\label{oud}
C_r(x)=\frac{2x\log r}{\log^2 x}+\Ocal\left(\frac{rx\log (er)}{\log^3 x}\right).
\end{equation*}
\end{Thm}
This result was generalized by Hashimoto \cite{hash} who determined the asymptotic 
behaviour
of $\# \left\{pq \leq x: \ p< q \leq f(p)\right\}$ for a large class of functions
$f$ satisfying $f(x)>x$. 
Another generalization was obtain by Justus \cite{Justus} who obtained an asymptotic
for $\# \left\{pq \leq x: \ p< q \leq x^{\theta}p\right\}$, with $0<\theta<1$ fixed.
On the more cryptographic side, there is the dissertation by Loebenberger \cite{Loebenberger}.\\
\indent The main aim of this paper is to establish a very precise asymptotic formula for
$C_r(x)$ (Corollary \ref{crcor}). On our way towards establishing 
this, we show that RSA-integers are rather unbiased (Corollary \ref{smallbias}). As
a particular case we obtain that the RSA-integer analogue of $r(x)$ shows little
bias (Corollary \ref{modi4}).\\
\indent As usual by $\pi(x)$ we denote the number of primes $p\le x$.
We will use the prime number theorem in the form
\begin{equation}
\label{PNT}
\pi(x)={\rm Li}(x)+\Ocal\Big(xe^{-c\sqrt{\log x}}\Big),
\end{equation}
where
$${\rm Li}(x)=\int_2^x \frac{dt}{\log t}$$
denotes the logarithmic integral.\\
\indent In our main result, a variant, $F_r(x)$, of the logarithmic integral will play the
main role. It is easily seen to be a \textbf{concave} function for $x\ge 2r$.
\begin{Thm}
\label{Thm3}
Let $r>1$ be an arbitrary fixed real number.
Given two sets of primes $S_1$ and $S_2$, we 
put $$D_r(x):=\# \{pq \leq x: p< q \leq rp,~p\in S_1,~q\in S_2\}.$$
Suppose that for $j=1,2$ the counting functions associated to $S_j$ satisfy 
\begin{equation}
\label{Sassumption}
\pi_{S_j}(x):=\sum_{p\le x\atop p\in S_j}1=\frac{1}{\delta_j}{\rm Li}(x)+
\Ocal\left(xe^{-c\sqrt{\log x}}\right),
\end{equation}
where $\delta_j>0$ and $c>0$ is a positive constant.
For $x\ge 2r$ put
$$F_r(x)=\int_{2r}^x\frac{\log \log \sqrt{rt}-\log \log \sqrt{t/r}}{\log t}dt.$$
For $x\ge 2r$ and $x$ tending to infinity we have
$$\delta_1\delta_2D_r(x)=F_r(x)+
\Ocal\left( rxe^{-c(\epsilon)\sqrt{\log x}}\right),$$
where $c(\epsilon)=(1-\epsilon)c/\sqrt{2}$ and $0<\epsilon<1$ is arbitrary.
\end{Thm}
\begin{cor}
\label{eerstecor}
We have
$C_r(x)=F_r(x)+\Ocal\left( rxe^{-c(\epsilon)\sqrt{\log x}}\right)$.
\end{cor}
\begin{proof}
For $S_1$ and $S_2$ we take the set of all primes. 
It follows by \eqref{PNT} that 
condition \eqref{Sassumption} is satisfied with $\delta_1=\delta_2=1$.
\end{proof}
\begin{cor}
\label{smallbias}
We have
$\delta_1\delta_2D_r(x)=C_r(x)+\Ocal\left(rxe^{-c(\epsilon)\sqrt{\log x}}\right)$ and
$$R_r(x) :=\frac{\delta_1\delta_2D_r(x)}{C_r(x)}=1+\Ocal_r\left((\log^2 x)\,\, e^{-c(\epsilon)\sqrt{\log x}}\right).$$
\end{cor}
\begin{proof}
Follows from Theorem \ref{Thm3}, Corollary \ref{eerstecor} and Theorem \ref{DM}.
\end{proof}
\begin{cor}
\label{modi4} Let $a_1,d_1,a_2,d_2$ be natural numbers with $(a_1,d_1)=(a_2,d_2)=1$.
We have
$$\frac{\#\{pq\le x:p\equiv a_1({\rm mod~}d_1),~q\equiv a_2({\rm mod~}d_2),~p<q\le rp\}}
{\#\{pq\le x:p<q\le rp\}/(\varphi(d_1)\varphi(d_2))}
=1+\Ocal_r\left((\log^2 x)\, \, e^{-c(\epsilon)\sqrt{\log x}}\right).$$
\end{cor}
\begin{proof}
For $S_j$ we take in Corollary \ref{smallbias} the set of all primes $\equiv a_j({\rm mod~}d_j)$. 
The prime number theorem for arithmetic progressions in the
form
\begin{equation}
\label{PNTAP}
\pi_{S_j}(x)=\frac{{\rm Li}(x)}{\varphi(d_j)} +\Ocal\left(x e^{-c\sqrt{\log x}}\right)
\end{equation}
then shows that condition \eqref{Sassumption} is satisfied with 
$\delta_j=\varphi(d_j)$ (as usual $\varphi(d)$ denotes Euler's totient function).
\end{proof}
On comparing \eqref{dumm} with Corollary \ref{smallbias} 
(or with Corollary \ref{modi4} for that matter) we see that for RSA-integers there is far less bias than for integers $n\le x$ having two
distinct prime factors.\\
\indent 
The implicit error terms of results involving the sets $S_j$ might
depend on them. For notational convenience this possible dependence is 
not explicitly indicated.

\subsection{Asymptotic formulas for $F_r(x)$ and $D_r(x)$}
By splitting the integration range in say $2$ to $\sqrt{x}$ and $\sqrt{x}$ to $x$, one sees
that 
\begin{equation}
\label{zoflauw}
\int_2^x \frac{dt}{\log^k t}=\Ocal_k
\left(\frac{x}{\log^{k}x}\right).
\end{equation}
Using this and partial integration we infer that for every
$n\ge 2$ we have 
\begin{equation}
\label{liexp}
{\rm Li}(x)=\sum_{k=1}^{n-1} (k-1)!\frac{x}{\log ^k x}+
\Ocal_n\Big(\frac{x}{\log^{n} x}\Big).
\end{equation}
Theorem \ref{expansie} provides the analogue of the 
asymptotic formula \eqref{liexp} for $F_r(x)$.
\begin{Thm}
\label{expansie}
Let $r>1$ be an arbitrary fixed real number and $n\ge 2$ an integer.
Then
\begin{equation*}
F_r(x)=\sum_{k=1}^{n-1}a_k(r)\frac{x}{\log^{k+1} x}+
\Ocal_{n}\left(\frac{x\log^{2\lfloor n/2\rfloor+1} (2r)\log r}{\log^{n+1} x}\right),
\end{equation*}
where
\begin{equation*}
a_k(r)=\sum_{j=1}^{[\frac{k+1}{2}]}\frac{k!}{(2j-1)!}\frac{2\log^{2j-1}r}{2j-1}, 
\end{equation*}
with $[x]$ the integral part of $x$.
\begin{table}[ht]
\centering
  \begin{tabular}{ |l || l | }
    \hline
    $k$ & $a_k(r)$  \\ \hline \hline
    $1$ & $2\rho$\\ \hline
    $2$ & $4\rho $ \\ \hline
    $3$ & $12\rho+2\rho^3/3$ \\ \hline
    $4$ & $48\rho+8\rho^3/3$ \\ \hline
    $5$ & $240\rho+40\rho^3/3+2\rho^5/5$  \\ \hline
    $6$ &   $1440 \rho + 80 \rho^3  + 12\rho^5/5$ \\ \hline
    $7$ &   $10080\rho + 560\rho^3   + 84\rho^5/5  + 2\rho^7/7 $\\ \hline
    $8$ &   $80640 \rho+ 4480 \rho^3  + 672\rho^5/5  + 16\rho^7/7$ \\ \hline
    $9$ &  $725760 \rho + 40320 \rho^3 + 6048\rho^5/5 + 144\rho^7/7  + 2\rho^9/9$ \\ \hline
    $10$ &  $7257600 \rho + 403200 \rho^3 + 12096 \rho^5  + 1440\rho^7/7  + 20\rho^9/9$ \\ \hline
  \end{tabular}
\label{Tab1}
\caption{The polynomial $a_k(r)$ for $k \in \{1, 2, \ldots, 10\}$ with $\rho=\log r$.}
\end{table}
\end{Thm}
Theorem \ref{expansie} when combined with Theorem \ref{Thm3} 
now yields Theorem \ref{Thm2}.

\begin{Thm}
\label{Thm2}
Let $S_1$ and $S_2$ be sets of primes satisfying the condition~\eqref{Sassumption}.
Let $r>1$ be an arbitrary fixed real number and $n\ge 2$ be an arbitrary integer. 
As $x$ tends to infinity, we have 
\begin{equation*}
\delta_1\delta_2D_r(x)=\sum_{k=1}^{n-1}a_k(r)\frac{x}{\log^{k+1} x}
+
\Ocal\left(rx e^{-c(\epsilon)\sqrt{\log x}}\right)+\Ocal_{n}\left(\frac{x\log^{2\lfloor n/2\rfloor+1} (2r)\log r}{\log^{n+1} x}\right),
\end{equation*}
where $c(\epsilon)$ and $a_k(r)$ are defined in Theorem \textrm{\ref{Thm3}}, respectively Theorem \textrm{\ref{expansie}}.
\end{Thm}
\begin{cor}
\label{crcor}
Let $r>1$
be an arbitrary fixed real number and $n\ge 2$ be an arbitrary integer.
As $x$ tends to infinity, we have 
\begin{equation*}
C_r(x)=\sum_{k=1}^{n-1}a_k(r)\frac{x}{\log^{k+1} x}
+
\Ocal\left(rx e^{-c(\epsilon)\sqrt{\log x}}\right)+\Ocal_{n}\left(\frac{x\log^{2\lfloor n/2\rfloor+1} (2r)\log r}{\log^{n+1} x}\right),
\end{equation*}
\end{cor}
\begin{cor}
Let $B>0$ be an arbitrary real number. Uniformly for $1<r\le \log^B x$ we have 
\begin{equation*}
C_r(x)=\sum_{k=1}^{n-1}a_k(r)\frac{x}{\log^{k+1} x}
+\Ocal_{n,B}\left(\frac{x}{\log^{n+1} x}(\log \log x)^{2\lfloor n/2\rfloor+2}\right),
\end{equation*}
\end{cor}
Note that Corollary \ref{crcor} with $n=2$ slightly improves on
Theorem \ref{DM}.
With more work it is possible to improve the error terms in our
results in the $r$ aspect. As this seems to be mathematically not very 
important, but requires considerable effort and is not beneficial for
the brevity and clarity of our presentation, we have abstained from pursuing this. 
Indeed, if we would have ignored the $r$ dependence altogether, our proofs 
would have been simpler and shorter. We point out that we want to have estimates valid for $r>1$, not just for 
$r$ large. E.g., it is true that $1+\log r=\Ocal(\log r)$ as $r$ tends to infinity, but
not if we take $r>1$. Correct in this case is $1+\log(r)=\log(er)=\Ocal(\log(2r))$.\\
\indent Our proof of Theorem \ref{Thm2} has Theorem \ref{Thm3}
as a starting point. 
We provide some more details of the 
proofs in Section \ref{sketch} followed by the full proofs 
in the remaining sections.




\section{Sketch of the proofs}
\label{sketch}
To understand the proofs it is helpful to  
first get an idea of the proof of Theorem \ref{DM}.\\
\indent For any prime $p$ we define $f_p(x)$ to be the number of primes $q$ such
that $pq\le x$ and $p<q\le rp$. An easy computation then yields
\begin{equation}
\label{eq8}
C_r(x)=\sum_{p\le x}f_p(x)=-\sum_{p\le \sqrt{x}}\pi(p)+\sum_{p\le \sqrt{x/r}}\pi(r p)
+\sum_{\sqrt{x/r}<p\le \sqrt{x}}\pi\left(\frac{x}{p}\right).
\end{equation}
The asymptotic behaviour of each of these three sums is then 
determined. As input not more than
the prime number theorem with error $\Ocal(x\log^{-3}x)$ is used (that is
the estimate~\eqref{liexp} with $n=3$).\\
\indent Our proof of Theorem \ref{Thm3} starts by noting
that (cf. the proof of \cite[Lemma 2]{Decker-Moree}) 
\begin{equation}
\label{analogue}
D_r(x)= -\sum\limits_{\substack{p\leq \sqrt{x}\\ p\in S_1}} 
\pi_{S_2}(p)
+ \sum\limits_{\substack{p\leq \sqrt{x/r}\\ p\in S_1}}\pi_{S_2}(rp)
+ \sum\limits_{\substack{\sqrt{x/r}<p\leq \sqrt{x}\\p\in S_1}}\pi_{S_2}\left(\frac{x}{p}\right).
\end{equation}
The first two sums in~\eqref{analogue} can be dealt with the same way since the second sum with $r=1$
is the negative of the first sum.
The idea is now to replace every $\pi_{S_2}(z)$ in~\eqref{analogue} by
a ${\rm Li}(z)/\delta_2$, thus producing a small error 
(by the assumption~\eqref{Sassumption}) and then to interchange
the order of integration and summation. In doing so terms
of the form $\pi_{S_1}(z)$ appear and those we replace by ${\rm Li}(z)/\delta_1$ (by 
assumption~\eqref{Sassumption} again at the cost of introducing a small error). We
thus obtain an approximation for $D_r(x)$ with main term $G_r(x)/(\delta_1\delta_2)$, 
where
\begin{equation}
\label{grrr}
G_r(x)=\frac{1}{2}{\rm Li}(\sqrt{x})^2-\int_{2r}^{\sqrt{rx}}\frac{{\rm Li}(t/r)}{\log t}dt+\int_{\sqrt{x}}^{\sqrt{rx}}\frac{{\rm Li}(x/t)}{\log t}dt.
\end{equation}
Taking the derivative of $G_r(x)$ with respect to $x$ then shows that
$G_r(x)=F_r(x)+\Ocal(r)$. This then completes the proof.\\
\indent In Theorem \ref{expansie} we try to 
obtain an expansion of the form  $\sum_{k=1}^{n-1}g_k(r)x\log^{-k-1}x$ 
for $G_r(x)$, with $g_k(r)$ to be determined. The key observation now is that $G_r'(x)=F_r'(x)$ is such
a simple function that an expansion of the form 
$\sum_{k=1}^{n-1}h_k(r)x\log^{-k-1}x$ for 
$G_r'(x)$ is easily found, where the $h_k(r)$ are readily determined. 
Subsequently one integrates this expansion termwise. This then 
shows that $g_k(r)=a_k(r)$ with $a_k(r)$ as defined in Theorem \ref{expansie}.
One has to take some care
to show that this termwise integration is actually allowed.\\
\indent Our first approach 
for establishing Theorem \ref{expansie}, 
was to substitute the expansion~\eqref{liexp} for 
Li$(z)$ in~\eqref{grrr}, leading to the conclusion that an 
expansion as in Theorem \ref{expansie}
exists. However, in this way complicated expressions for 
the polynomials 
$a_k(r)$ are obtained.
On computing various examples of those using 
Mathematica and studying the $j$-th coefficient
of $a_k(r)$ as a sequence using the On-Line Encyclopedia of Integer Sequences (OEIS), 
we made an explicit conjecture for the coefficients of $a_k(r)$ 
and eventually proved 
it by quite a different route.


\section{Proof of Theorem~\ref{Thm3}}
\subsection{Some lemmas}
In the analysis of the error term of our result, we make use of the following easy
estimates. The ones in part a) arise on replacing terms of the form $\pi_{S_2}(z)$
by ${\rm Li}(z)/\delta_2$, the ones in part b) on replacing terms of the form $\pi_{S_1}(z)$
by ${\rm Li}(z)/\delta_1$.
\begin{lem}
\label{lem10} Let $c>0$, $r>1$ and $0<\epsilon<1$. 
Put $c(\epsilon)=(1-\epsilon)c/\sqrt{2}$.\\
{\rm a)} We have
\begin{equation}
\label{eq15}
\sum_{p\leq \sqrt{x/r}}p e^{-c\sqrt{\log (rp)}}
\le \sum_{p\leq \sqrt{x}}p e^{-c\sqrt{\log p}}=\Ocal\left( xe^{-c(\epsilon)\sqrt{\log x}}\right)
\end{equation}
and
\begin{equation}
\label{eq16}
\sum_{\sqrt{x/r}<p\leq \sqrt{x}}\frac{x}{p}e^{-c\sqrt{\log (x/p)}}=\Ocal\left( xe^{-c(\epsilon)\sqrt{\log x}}\right).
\end{equation}
{\rm b)} The estimates~\eqref{eq15} and  \eqref{eq16} also hold true if we replace the sum
by an integral over the same range and $p$ by a continuous variable.
\end{lem}
\begin{proof}
We only prove part a), the proof of b) being similar.\\
\indent The first inequality is obvious.
Now notice that
\begin{equation*}
\sum_{p\leq \sqrt{x}}p e^{-c\sqrt{\log p}}
\leq 
\sum_{p\leq x^{\frac{1}{2}(1-\epsilon)^2}}p+\sum_{x^{\frac{1}{2}(1-\epsilon)^2}<p\leq \sqrt{x}}pe^{-c(\epsilon)\sqrt{\log x}}=\Ocal\left( xe^{-c(\epsilon)\sqrt{\log x}}\right).
\end{equation*}
The proof of estimate 
\eqref{eq16} follows immediately from the observation
\begin{equation*}
\sum_{\sqrt{x/r}<p\leq \sqrt{x}}\frac{x}{p}e^{-c\sqrt{\log (x/p)}}
\leq 
xe^{-\frac{c}{\sqrt{2}}\sqrt{\log x}}\sum_{p\leq \sqrt{x}}\frac{1}{p}
=
\Ocal\left( xe^{-c(\epsilon)\sqrt{\log x}}\right),
\end{equation*}
where we used that $\sum_{p\le z}p^{-1}=\Ocal(\log \log z)$.
\end{proof}
The sums in the next two lemmas arise on replacing $\pi_{S_1}$ by Li in the
second and third sum as appearing in (\ref{analogue}).
\begin{lem}
\label{sumoneandtwo}
Let $r\ge 1$ be an arbitrary fixed real number and $S_1$ any set of primes. Then
\begin{eqnarray*}
\sum_{p\leq \sqrt{x/r}\atop p\in S_1}{\rm Li}(rp)
&=& 
\pi_{S_1}(\sqrt{x/r}){\rm Li}(\sqrt{rx})
-
\int\limits_{2}^{\sqrt{rx}}\frac{\pi_{S_1}(t/r)}{\log t}\, dt.
\end{eqnarray*}
\end{lem}
\begin{proof}
We find that
\begin{eqnarray*}
\sum_{p\leq \sqrt{x/r}\atop p\in S_1}{\rm Li}(rp)
= 
\sum_{p\leq \sqrt{x/r}\atop p\in S_1}\int\limits_{2}^{rp}\frac{dt}{\log t}
=
\int\limits_{2}^{\sqrt{rx}}\frac{A_{S_1}(t)}{\log t}\, dt,
\end{eqnarray*}
where $A_{S_1}(t)= \#\{p\leq \sqrt{x/r} \, : \, rp\geq t,~p\in S_1 \}$. 
The result easily follows on noting 
that 
$A_{S_1}(t) = \pi_{S_1}(\sqrt{x/r})-\pi_{S_1}(t/r)$ for $2\le t\le \sqrt{rx}$.
\end{proof}
\begin{lem}
\label{sumthree}
Let $r>1$ be an arbitrary fixed real number and $S_1$ any set of primes. Then
\begin{equation*}
\sum_{\sqrt{x/r}<p\leq \sqrt{x}\atop p\in S_1}{\rm Li}\left( \frac{x}{p}\right)= 
\pi_{S_1}(\sqrt{x}){\rm Li}(\sqrt{x})
-\pi_{S_1}(\sqrt{x/r}){\rm Li}( \sqrt{rx})
+\int\limits_{\sqrt{x}}^{\sqrt{rx}}\frac{\pi_{S_1}(x/t )}{\log t}\, dt
\end{equation*}
\end{lem}
\begin{proof}
Note that
\begin{equation}
\sum_{\sqrt{x/r}<p\leq \sqrt{x}\atop p\in S_1}{\rm Li}\left( \frac{x}{p}\right)
= \sum_{\sqrt{x/r}<p\leq \sqrt{x}\atop p\in S_1}\int\limits_{2}^{x/p}\frac{dt}{\log t}
=
\int\limits_{2}^{\sqrt{rx}}\frac{B_{S_1}(t)}{\log t}\, dt,
\end{equation}
where
$B_{S_1}(t)= \# \{ \sqrt{x/r}<p \leq \sqrt{x} \, : \, 2\leq t\leq x/p,~p\in S_1 \}$.
Clearly
\begin{equation*}
B_{S_1}(t)=
\begin{cases}
  \displaystyle{\pi_{S_1}(\sqrt{x})-\pi_{S_1}( \sqrt{x/r})} \quad  &\textrm{if $ 2\leq t\leq \sqrt{x}$}; \\
  \displaystyle{\pi_{S_1}(x/t)-\pi_{S_1}( \sqrt{x/r})} \quad  &\textrm{if $\sqrt{x}<t\leq \sqrt{rx} $}.
\end{cases}
\end{equation*}
From this the result easily follows.
\end{proof}
Alternatively the lemmas \ref{sumoneandtwo} and \ref{sumthree} can be  
also proved by using partial integration and making
an obvious linear transformation in the resulting integral. (The details are left to the
interested reader.)
\subsection{Proof Theorem~\ref{Thm3}}
Recall that, for $r>1$, $G_r(x)$ is defined by 
\begin{equation}
\label{fr(x)}
G_r(x)=\frac{1}{2}{\rm Li}(\sqrt{x})^2-\int_{2r}^{\sqrt{rx}}\frac{{\rm Li}(t/r)}{\log t}dt+\int_{\sqrt{x}}^{\sqrt{rx}}\frac{{\rm Li}(x/t)}{\log t}dt.
\end{equation}
In our proof of Theorem~\ref{Thm3} we will make use of the following observation.
\begin{lem}
\label{grderivative}
We have $G_r'(x)=\frac{1}{\log x} \left( \log \log \sqrt{rx}-\log \log \sqrt{x/r}\right).$
\end{lem}
\begin{proof}
The derivative of the first term on the right hand side of~\eqref{fr(x)} is 
\begin{equation*}
\frac{d}{dx} \left( \frac{1}{2}{\rm Li}(\sqrt{x})^2\right)=\frac{{\rm Li}(\sqrt{x})}{\sqrt{x}\log x }.
\end{equation*}
The derivative of the first and the second integral on the right hand side 
of~\eqref{fr(x)} equals
\begin{equation*}
\frac{d}{dx} \left(\int_{2r}^{\sqrt{rx}}\frac{{\rm Li}(t/r)}{\log t}dt\right)
=
\frac{{\rm Li}(\sqrt{x/r})}{2\sqrt{x/r}\log \sqrt{rx}},
\end{equation*}
respectively 
\begin{equation*}
\frac{d}{dx}\left(\int_{\sqrt{x}}^{\sqrt{rx}}\frac{{\rm Li}(x/t)}{\log t}dt\right)
=
\frac{{\rm Li}(\sqrt{x/r})}{2\sqrt{x/r}\log \sqrt{rx}}
-\frac{{\rm Li}(\sqrt{x})}{\sqrt{x}\log x }
+
\int_{\sqrt{x}}^{\sqrt{rx}}\frac{dt}{t \log t \log (x/t)}.
\end{equation*}
On adding them we get 
\begin{equation*}
{G^\prime}_r(x)= \int_{\sqrt{x}}^{\sqrt{rx}}\frac{dt}{t \log t \log (x/t)}.
\end{equation*}
Note that
\begin{equation*}
G_r'(x)=
\frac{1}{\log x} \int_{\sqrt{x}}^{\sqrt{rx}}\frac{dt}{t\log t}+\frac{1}{\log x} \int_{\sqrt{x}}^{\sqrt{rx}}\frac{dt}{t\log (x/t)}.
\end{equation*}
By making a simple change of variable $t=x/v$ in the second integral on the right hand side above, we 
obtain
\begin{equation*}
G_r'(x)=
\frac{1}{\log x} \int_{\sqrt{x/r}}^{\sqrt{rx}}\frac{dv}{v\log v}=\frac{1}{\log x} \left( \log \log \sqrt{rx}-\log \log \sqrt{x/r}\right),
\end{equation*}
thus concluding the proof.
\end{proof}

\begin{proof}[Proof of Theorem {\rm \ref{Thm3}}]
In~\eqref{analogue} we replace every
term $\pi_{S_2}(z)$ by the estimate given in (\ref{Sassumption}) 
and invoke Lemma \ref{lem10} a) to bound the resulting sums of 
error estimates giving rise to the asymptotic formula
\begin{equation}
\label{eq9}
\delta_2D_r(x)= -\sum\limits_{\substack{p\leq \sqrt{x}\\ p\in S_1}} 
{\rm Li}(p)
+ \sum\limits_{\substack{p\leq \sqrt{x/r}\\ p\in S_1}}{\rm Li}(rp)
+ \sum\limits_{\substack{\sqrt{x/r}<p\leq \sqrt{x}\\p\in S_1}}{\rm Li}\left(\frac{x}{p}\right)+\Ocal\left(rx e^{-c(\epsilon)\sqrt{\log x}}\right).  
\end{equation}
From~\eqref{eq9}, Lemmas \ref{sumoneandtwo}, \ref{sumthree} 
and the observation that $\pi_{S_1}(z)=0$ for $z<2$
we infer that
$$\delta_2 D_r(x)=\int_2^{\sqrt{x}}\frac{\pi_{S_1}(t)}{\log t}dt-\int_{2r}^{\sqrt{rx}}\frac{\pi_{S_1}(t/r)}{\log t}dt+\int_{\sqrt{x}}^{\sqrt{rx}}\frac{\pi_{S_1}(x/t)}{\log t}dt
+\Ocal\left(rx e^{-c(\epsilon)\sqrt{\log x}}\right).$$
By partial integration, 
\begin{equation}
\label{simpleintegration}
\int_2^{\sqrt{x}}\frac{{\rm Li}(t)}{\log t}dt=\frac{1}{2}{\rm Li}(\sqrt{x})^2.
\end{equation}
Using this we see that if in the three integrals 
appearing in~\eqref{eq9} we replace $\pi_{S_1}$
by $\delta_1^{-1}{\rm Li}$ we obtain $G_r(x)/\delta_1$. Using Lemma \ref{lem10} b) we estimate
the sum of the errors made on making this replacement and conclude 
that as $x$ tends to infinity we have
\begin{equation}
\label{vorige}
\delta_1\delta_2D_r(x)=G_r(x)+
\Ocal\left( rxe^{-c(\epsilon)\sqrt{\log x}}\right).
\end{equation}
Using Lemma \ref{grderivative} 
we notice that
$G_r(x)-G_r(2r)=F_r(x)$ for $x\ge 2r$. 
Using some rough estimates on finds that $G_r(2r)=O(r)$ and
hence we infer that $G_r(x)=F_r(x)+O(r)$.
The proof is 
concluded on inserting this estimate in~\eqref{vorige}. 
\end{proof}
\section{Proof of Theorem \ref{Thm2}}
We will make use of the following lemma.
\begin{lem}
\label{stoptnooit}
Let $\mathfrak{b}=\{b_j\}_{j=1}^{\infty}$ be a sequence of 
non-negative real numbers and $n\ge 2$ an arbitrary integer.
We define
\begin{equation}
\label{Nb}
N_{\mathfrak{b}}(x):=\sum_{j=1}^{n-1}b_j\int_2^x\frac{dt}{\log ^j t}.
\end{equation}
As $x$ tends to infinity we have
$$N_{\mathfrak{b}}(x)=\sum_{k=1}^{n-1}\Big(\sum_{j=1}^{k}b_j
\frac{(k-1)!}{(j-1)!}\Big)\frac{x}{\log ^k x}+\Ocal_n\left(\frac{xB_n}{\log ^n x}\right),$$
where $B_n=\sum\limits_{j=1}^n b_j$.
\end{lem}
\begin{proof}
By partial integration one finds, with $n_j\ge 1$ an arbitrary integer,
\begin{equation*}
\int_2^x \frac{dt}{\log^j t}=\sum_{m=1}^{n_j} \frac{(j+m-2)!}{(j-1)!}
\frac{x}{\log^{j+m-1}x}+\Ocal_{n_j}\left(\frac{x}{\log^{j+n_j} x}\right).
\end{equation*}
For $j=1,\ldots,n-1$ we insert this in~\eqref{Nb} and take, e.g., $n_j=n+1-j$. Rearranging
terms then yields the result.
\end{proof}
\begin{proof}[Proof of Theorem {\rm \ref{Thm2}}]
Let $ |u|<1. $ Using the Taylor series
\[\log (1-u)=-\sum_{\ell=1}^{\infty}\frac{u^\ell}{\ell},\]
we conclude that
\begin{equation}
\label{taylortje}
\log\left( \frac{1+u}{1-u} \right)=2\sum_{\ell=1}^{\infty}\frac{u^{2\ell-1}}{2\ell-1}. 
\end{equation}
Define $$E_m(u)= \log\left( \frac{1+u}{1-u} \right)-2\sum_{\ell=1}^{m}\frac{u^{2\ell-1}}{2\ell-1}.$$
Note that 
\begin{equation}
\label{emuestimate}
0<E_m(u)=\frac{2u^{2m+1}}{2m+1}+\frac{2u^{2m+3}}{2m+3}+\cdots< 
2\sum_{k=1}^{\infty}u^{2m+2k-1}=\frac{2u^{2m+1}}{1-u^2}\quad\text{for~}0<u<1.
\end{equation}
Clearly
$$F_r'(x)=\frac{\log \log \sqrt{rx}-\log \log \sqrt{x/r}}{\log x}
=\frac{1}{\log x} \left(\log\left(1+\frac{\log r}{\log x}\right)-\log\left(1-\frac{\log r}{\log x}\right)\right).$$
Recall that by assumption $x\ge 2r$. 
For those $x$ we find by~\eqref{taylortje} the Taylor series
$$F_r'(x)=
\frac{1}{\log x} \sum_{\ell=1}^{\infty}\frac{2}{2\ell-1} \left(\frac{\log r}{\log x} \right)^{2\ell-1}.$$
Using the definition of $E_m(u)$ it now follows that
\begin{equation}
\label{benerbijna3}
F_r(x)=\int_{2r}^x F_r'(t)dt =\sum_{\ell=1}^m\frac{2}{2\ell-1}\log^{2\ell-1}r\int_{2r}^x \frac{dt}{\log^{2\ell}t}
+\int_{2r}^{x}\frac{1}{\log t}E_m\left(\frac{\log r}{\log t}\right)dt.
\end{equation}
{}From~\eqref{emuestimate} we infer that, for $ x\ge 2r, $
\begin{equation*}
0\le\int_{2r}^{x}\frac{1}{\log t}E_m\left( \frac{\log r}{\log t}\right)dt<\frac{2\log^{2m+1}r}{1-\left( \frac{\log r}{\log (2r)}\right) ^2}\int_{2r}^{x}
\frac{dt}{\log^{2m+2}t}.
\end{equation*}
Here, we note that 
\begin{equation*}
\frac{1}{1-\left( \frac{\log r}{\log (2r)}\right) ^2}=\frac{\log^2 (2r)}{(\log 2)\log (2r^2)}=\Ocal\left( \log (2r)\right).
\end{equation*}
We conclude that
\begin{equation}
F_r(x)= \sum_{\ell=1}^m\frac{2}{2\ell-1}\log^{2\ell-1}r\int_{2r}^x \frac{dt}{\log^{2\ell}t}
+\Ocal_m\left( \frac{x\log (2r)\log^{2m+1}r}{\log^{2m+2}x}\right),
\end{equation}
which can be rewritten as
\begin{equation*}
F_r(x) =\sum_{\ell=1}^m\frac{2}{2\ell-1}\log^{2\ell-1}r\int_{2}^x \frac{dt}{\log^{2\ell}t}
+\Ocal_m\left(\frac{r}{\log(2r)}\right)+\Ocal_m\left( \frac{x\log (2r)\log^{2m+1}r}{\log^{2m+2}x}\right),
\end{equation*}
where we used \eqref{zoflauw} to estimate $\int_2^{2r}dt/\log^{2\ell}t$. On
noting that $r(\log(2r))^{-2}\log^{-2m-1}r$ is eventually increasing in
$r$ and $r\le x/2$ we see that
$$\frac{r}{\log(2r)}=\Ocal_m\left( \frac{x\log (2r)\log^{2m+1}r}{\log^{2m+2}x}\right),$$
and therefore we have
\begin{equation}
\label{benerbijna}
F_r(x) =\sum_{\ell=1}^m\frac{2}{2\ell-1}\log^{2\ell-1}r\int_{2}^x \frac{dt}{\log^{2\ell}t}+\Ocal_m\left( \frac{x\log (2r)\log^{2m+1}r}{\log^{2m+2}x}\right),
\end{equation}
Lemma \ref{stoptnooit} applied with
$$
b_j=\begin{cases}
\frac{2}{j-1}\log^{j-1}r & \textrm{if $j$ is even};\\
0 & \textrm{otherwise,}
\end{cases}
$$
gives 
\begin{equation}
\label{spuugzat}
\sum_{\ell=1}^m\frac{2}{2\ell-1}\log^{2\ell-1}r\int_{2}^x \frac{dt}{\log^{2\ell}t}=\sum_{k=1}^{2m}v_k(r)\frac{x}{\log^{k+1} x}+
\Ocal_m\left( \frac{x(\log r)\log^{2m-2}(2r)}{\log^{2m+2}x}\right),
\end{equation}
where 
$$v_k(r)=\sum_{j=1}^{[\frac{k+1}{2}]}\frac{k!}{(2j-1)!}\frac{2\log^{2j-1}r}{2j-1}
=a_k(r).$$
On combining \eqref{spuugzat} with \eqref{benerbijna} the proof is then easily completed in 
case $n=2m+1$ is odd. (Observe that the error term in \eqref{spuugzat} is
majorized by the one in \eqref{benerbijna}.)\\
\indent On noting that $a_{2m}(r)$ is an odd polynomial in $\log r$, we
see that for all $r>1$ we have $a_{2m}(r)=\Ocal_m((\log r)\log^{2m-2}(2r))$.
Therefore
$$\frac{a_{2m}(r)x}{\log^{2m+1}x}=\Ocal_m\Big(x\frac{(\log r)\log^{2m-2}(2r)}{\log^{2m+1}x}\Big),$$
and it follows from \eqref{spuugzat} that 
\begin{equation}
\label{spuugzat2}
\sum_{\ell=1}^m\frac{2}{2\ell-1}\log^{2\ell-1}r\int_{2}^x \frac{dt}{\log^{2\ell}t}=\sum_{k=1}^{2m-1}a_k(r)\frac{x}{\log^{k+1} x}+
\Ocal_m\left( \frac{x(\log r)\log^{2m-2}(2r)}{\log^{2m+1}x}\right).
\end{equation}
On combining \eqref{spuugzat2} with \eqref{benerbijna} the proof 
is then also completed in the 
remaining case where $n=2m$ is even. 
\end{proof}
\subsection{Integrality of the coefficients of the polynomial $a_k(r)$}
Recall that 
$$
a_k(r)=\sum_{j=1}^{[(k+1)/2]}a_{k,j}\log^{2j-1} r,
$$
with 
$$a_{k,j}=\frac{k!2}{(2j-1)!(2j-1)}.$$
On being confronted with Table 1 the reader might wonder about the integrality of the
coefficients $a_{k,j}$. The following result is easy to prove.
\begin{prop} Define $\mu(j)=\min\{k\ge 2j-1:(2j-1)!(2j-1)|k!\}$.\\
{\rm a)} We have $\mu(j)\le 4j-2$ with equality if and only if $2j-1$ is a prime number.\\
{\rm b)} The coefficient $a_{k,j}$ is an integer if and only if $k\ge \mu(j)$.\\
{\rm c)} Suppose that $2j-1=\prod_{p|2j-1}p^{e_p}$, with all exponents $e_p\le p$.
Then $$\mu(j)=2j-1+\max\{e_pp:p|2j-1\}.$$
\end{prop}
\section{Bias in the sense of Chebyshev}
Let $\pi(x;d,a)$ denote the number of primes $p\le x$ that satisfy $p\equiv a({\rm mod~}d)$. 
We restrict ourselves to the case where $a$ and $d$ are coprime, the cases where the
residue class modulo $d$ is said to be primitive. It is the only relevant
case here as the non-primitive residue classes have only
finitely many primes in them.
It is a consequence of Legendre's theorem from 1837 that the primes are equidistributed over
the primitive residue classes modulo $d$. Nevertheless, certain differences
of the form $\pi(x;d,a_1)-\pi(x;d,a_2)$ are positive for many
values of $x$ (where ``many" is best
quantified using a logarithmic measure). This phenomenon was first observed and studied
by Chebyshev who found that there is a strong bias for primes to be $\equiv 3({\rm mod~}4)$
rather than $\equiv 1({\rm mod~}4)$. For a survey see Granville and
Martin \cite{GM}.\\
\indent Recently Ford and Sneed \cite{Ford-Sneed} 
and Xiangchang Meng \cite{XM}
considered bias for products of two, respectively
$k$
primes, with $k\ge 2$ and fixed.
\begin{Problem}
Study the Chebyshev bias phenomenon for products of two proportional
primes.
\end{Problem}
Here especially the case where the modulus $d=10$ is of relevance.

\subsection*{Acknowledgements.} 
A large portion of this paper was written during 
the second author's stay in May 2016 
at the Max Planck Institute for Mathematics (MPIM). She would like to thank Pieter Moree for inviting her and gratefully acknowledges support, hospitality 
as well as the excellent environment for collaboration at the 
MPIM. She is supported by the Austrian Science Fund (FWF): Project F5507-N26, which is a part of 
the Special Research Program ``Quasi Monte Carlo Methods: Theory and Applications".\\
\indent We would like to thank Florian Luca, Igor Shparlinski and G\'erald Tenenbaum for
helpful feedback. Furthermore we thank Alexandru Ciolan and Kate Kattegat for proofreading 
and Alex Weisse for his 
TeXnical help. The first author also acknowledges 
fruitful discussions with Yara Elias and
Wadim Zudilin.

\medskip\noindent Pieter Moree \par\noindent
{\footnotesize Max-Planck-Institut f\"ur Mathematik,
Vivatsgasse 7, D-53111 Bonn, Germany.\hfil\break
e-mail: {\tt moree@mpim-bonn.mpg.de}}

\medskip\noindent Sumaia Saad Eddin \par\noindent
{\footnotesize Institute of Financial Mathematics and Applied Number Theory, Altenbergerstrasse 69, 4040 Linz, Austria.
\hfil\break
e-mail: {\tt sumaia.saad\_eddin@jku.at}}

\end{document}